\documentclass[11pt]{article}

\usepackage{setspace}
\usepackage{amsmath}
\usepackage{amsfonts}
\usepackage{amssymb}
\usepackage{graphicx}
\usepackage{amsthm}
\usepackage{enumerate}
\usepackage{color}
\usepackage{MnSymbol}
\usepackage{fullpage}

\newtheorem{theorem}{Theorem}
\newtheorem{lemma}[theorem]{Lemma}
\newtheorem{definition}[theorem]{Definition}
\newtheorem{proposition}[theorem]{Proposition}
\newtheorem{corollary}[theorem]{Corollary}
\newtheorem{remark}[theorem]{Remark}

\title{Automorphism groups of Cayley-Dickson loops}
\author{Jenya Kirshtein \thanks{ykirshte@du.edu, Department of Mathematics, University of Denver, 2360 South Gaylord street, Denver, CO 80208, USA}}
\date{\vspace{-5ex}}
\begin{document}

\maketitle

\begin{abstract}
The Cayley-Dickson loop $Q_{n}$ is the multiplicative closure of basic elements of the algebra constructed by $n$ applications of the Cayley-Dickson doubling process (the first few examples of such algebras are real numbers, complex numbers, quaternions, octonions, sedenions). We discuss properties of the Cayley-Dickson loops, show that these loops are Hamiltonian, and describe the structure of their automorphism groups.

\end{abstract}

\let\thefootnote\relax 
\let\thefootnote\relax
\section{The Cayley-Dickson doubling process}\label{sec:intro}
The Cayley-Dickson doubling produces a sequence of power-associative algebras over a field. The dimension of the algebra doubles at each step of the construction. We consider the construction on~$\mathbb{R}$, the field of real numbers. The results of the paper hold for any field of characteristic other than $2$.\\
Let $\mathbb{A}_{0}=\mathbb{R}$ with conjugation $a^{*}=a$ for all $a\in \mathbb{R}$. Let $\mathbb{A}_{n+1}=\{(a,b)\left|\right.a,b\in A_{n}\}$ for $n\in \mathbb{N}$, where multiplication, addition, and conjugation are defined as follows:
\begin{eqnarray}
\label{eqn:mult} (a,b)(c,d) & = & (ac-d^{*}b,da+bc^{*}), \\ 
\label{eqn:add} (a,b)+(c,d) & = & (a+c,b+d), \\ 
\label{eqn:conj} (a,b)^{\ast} & = & (a^{\ast},-b). 
\end{eqnarray}
Conjugation defines a norm $\left\|a\right\|=\left(aa^{*}\right)^{1/2}$ and the multiplicative inverse for nonzero elements\\ $a^{-1}=a^{*}\slash\left\|a\right\|^2$. Notice that $(a,b)(a,b)^{*} =(\left\|a\right\|^{2}+\left\|b\right\|^2,0)$ and $(a^{*})^{*}=a$.  Dimension of $\mathbb{A}_{n}$ over $\mathbb{R}$ is~$2^{n}$.
\begin{definition}\label{def:divalg}
A nontrivial algebra A over a field is a \emph{division algebra} if for any nonzero $a\in A$ and any $b\in A$ there is a unique $x\in A$ such that $ax=b$ and a unique $y\in A$ such that $ya=b$.  
\end{definition}
\begin{definition}\label{def:ndivalg}
A \emph{normed division algebra} A is a division algebra over the real or complex numbers which is a normed vector space, with norm $\left\|\cdot\right\|$ satisfying $\left\|xy\right\|=\left\|x\right\|\left\|y\right\|$ for all $x,y\in A$.
\end{definition}
\begin{theorem}[Hurwitz, 1898 \cite{Hurwitz:1898}]\label{thm:hurwitz}
The only normed division algebras over $\mathbb{R}$ are $\mathbb{A}_{0}=\mathbb{R}$ (real numbers), $\mathbb{A}_{1}=\mathbb{C}$ (complex numbers), $\mathbb{A}_{2}=\mathbb{H}$ (quaternions) and $\mathbb{A}_{3}=\mathbb{O}$ (octonions).
\end{theorem}
\section{Cayley-Dickson loops and their properties}\label{sec:loops}
We will consider multiplicative structures that arise from the Cayley-Dickson doubling process.
\begin{definition}\label{def:loop}
A \emph{loop} is a nonempty set $L$ with binary operation $\cdot$ such that
\begin{enumerate}
	\item there is a neutral element $1\in L$ such that $1\cdot x=x\cdot 1=x$ for all $x\in L$, 
	\item for all $x,z\in L$ there is a unique $y$ such that $x\cdot y=z$,
	\item for all $y,z\in L$ there is a unique $x$ such that $x\cdot y=z$.
\end{enumerate}
\end{definition}
Define Cayley-Dickson loops $\left(Q_{n},\cdot \right)$ inductively as follows: 
\begin{eqnarray*}\label{eqn:CD1}
 Q_{0} & = & \{\pm(1)\}, Q_{1} = \{\pm(1,0),\pm(1,1)\},\\ 
 Q_{n} & = & \{\pm(x_{1},x_{2},\ldots,x_{n},0),\pm(x_{1},x_{2},\ldots,x_{n},1)\left|\right.\pm(x_{1},x_{2},\ldots,x_{n})\in Q_{n-1}\}, n\in \mathbb{N}. 
\end{eqnarray*}
In a compact form, 
\begin{equation}\label{eqn:CD2}
 Q_{0}=\{\pm(1)\},\ \  Q_{n}=\{\pm(x,0),\pm(x,1)\left|\right.\pm x\in Q_{n-1}\}.
\end{equation}
Using this approach, multiplication (\ref{eqn:mult}) becomes 
\begin{eqnarray} 
\label{eqn:mult0} (x,0)(y,0) & = & (xy,0),\\
\label{eqn:mult1} (x,0)(y,1) & = & (yx,1), \\
\label{eqn:mult2} (x,1)(y,0) & = & (xy^{*},1),\\
(x,1)(y,1) & = & (-y^{*}x,0). 
\end{eqnarray}
Conjugation (\ref{eqn:conj}) modifies to
\begin{eqnarray} \label{eqn:conj1}
(x,0)^{*} & = & (x^{*},0),\\ 
(x,1)^{*} & = & (-x,1). 
\end{eqnarray}
All elements of $Q_{n}$ have norm one due to the fact that 
\begin{equation*} \label{eqn:norm1}
\left\|(x,x_{n+1})\right\| = \left\|x\right\|=\left\|(x_{1},\ldots,x_{n})\right\|=\ldots=\left\|x_{1}\right\|=1, 
\end{equation*} 
however, not all the elements of $\mathbb{A}_{n}$ of norm one are in $Q_{n}$. The Cayley-Dickson loop is the multiplicative closure of basic elements of the corresponding Cayley-Dickson algebra. The first few examples of the Cayley-Dickson loops are the group of real units $\mathbb{R}_{2}$ (abelian); the group of complex integral units $\mathbb{C}_{4}$ (abelian); the group of quaternion integral units $\mathbb{H}_{8}$ (not abelian); the octonion loop $\mathbb{O}_{16}$ (Moufang); the sedenion loop $\mathbb{S}_{32}$ (not Moufang); the trigintaduonion loop $\mathbb{T}_{64}$. \\
We write $Q_{n}$ or $Q$ instead of $\left(Q_{n},\cdot \right)$ further in the text. \\
Denote the loop generated by elements $x_{1},\ldots,x_{n}$ of a loop $L$ by $\left\langle x_{1},\ldots,x_{n}\right\rangle$. Denote by $i_{n}$ an element $(1_{Q_{n-1}},1)$ of $Q_{n}$. Such element $i_n$ satisfies $Q_{n}=\left\langle Q_{n-1},i_{n}\right\rangle$, thus $Q_{n}=\left\langle i_{1},i_{2},\ldots, i_{n}\right\rangle$. We call $i_{1},i_{2},\ldots, i_{n}$ the \emph{canonical generators} of $Q_{n}$. Any $x\in Q_{n}$ can be written as 
\begin{equation*}\label{eqn:x}
x=\pm\prod_{j=1}^n i_{j}^{\epsilon_{j}},\ \ \epsilon_{j}\in \{0,1\}.
\end{equation*}
For example,   
\begin{eqnarray*}
 Q_0=\mathbb{R}_{2} & = & \{1,-1\},\\
 Q_1=\mathbb{C}_{4} & = & \{(1,0),-(1,0),(1,1),-(1,1)\}=\left\langle  i_1\right\rangle = \{1,-1,i_1,-i_1\},\\
 Q_2=\mathbb{H}_{8} & = & \pm \{(1,0,0),(1,1,0),(1,0,1),(1,1,1)\} = \left\langle i_1,i_2\right\rangle = \pm \{1,i_{1},i_{2},i_{1}i_{2}\}.                                
\end{eqnarray*}
Next, we show some properties of the Cayley-Dickson loops.
\begin{theorem} [\cite{Culbert:07}]\label{thm:culbert}
Any pair of elements of a Cayley-Dickson loop generates a subgroup of the quaternion group. In particular, a pair $x,y$ generates a real group when $x=\pm1$ and $y=\pm1$; a complex group when either $x=\pm1$, or $y=\pm1$ (but not both), or $x=\pm y\neq\pm 1$; a quaternion group otherwise. 
\end{theorem}
Lemma~\ref{lemma:quasioct} extends Theorem~\ref{thm:culbert} and shows that any three elements of a Cayley-Dickson loop generate a subloop of either the octonion loop, or the quasioctonion loop.
\begin{definition}\label{def:diassoc}
A loop~$L$ is \emph{diassociative} if every pair of elements of $L$ generates a group in~$L$. 
\end{definition}
\begin{corollary}\label{cor:diassoc}
Every Cayley-Dickson loop is diassociative.
\end{corollary}
\begin{proof}
The quaternion group $\mathbb{H}_{8}$ is associative and the rest follows from Theorem \ref{thm:culbert}.
\end{proof}
\begin{definition}\label{def:commutant}
\emph{Commutant} of a loop L, denoted by C(L), is the set of elements that commute with every element of L. More precisely, $C(L)=$\{$a \left| \right. ax=xa$, $\forall x\in L$\}.
\end{definition}
\begin{definition}\label{def:nucleus}
\emph{Nucleus} of a loop L, denoted by N(L), is the set of elements that associate with all elements of L. More precisely, $N(L)=$\{$a \left| \right. a\cdot xy=ax\cdot y, xa\cdot y=x\cdot ay, xy\cdot a=x\cdot ya$, $\forall x,y\in L$\}.
\end{definition}
\begin{definition}\label{def:center}
\emph{Center} of a loop L, denoted by Z(L), is the set of elements that commute and associate with every element of L. More precisely, $Z(L)=C(L)\cap N(L)$.
\end{definition}
\begin{definition}[\cite{Pflugfelder:90} p.13]\label{def:normal}
Let S be a subloop of a loop L. Then S is called a \emph{normal subloop} if for all $x,y \in L$
\begin{align}
	\nonumber xS&=Sx,\\
	\nonumber (xS)y&=x(Sy),\\
	\nonumber x(yS)&=(xy)S.
\end{align}
\end{definition}
\begin{definition}\label{def:assoc_sub}
\emph{Associator subloop} of a loop $L$, denoted by $A(L)$, is the smallest normal subloop of L such that $L/A(L)$ is a group.
\end{definition}
\begin{definition}\label{def:derived_sub}
\emph{Derived subloop} of a loop $L$, denoted by $L'$, is the smallest normal subloop of L such that $L/L'$ is an abelian group.
\end{definition} 
\begin{lemma}\label{lemma:center}
Let S be a subloop of $Q_{n}$. The following holds
\begin{enumerate}
\item Center of $S$, $Z(S)=\left\{1,-1\right\}$ when $\left|S\right|>4$ and $Z\left(S\right)=S$ otherwise. 
\item Associator subloop of $S$, $A(S)=Z(S)$ when $\left|S\right|>8$ and $A(S)=1$ otherwise. 
\item Derived subloop of $S$, $S'=Z(S)$ when $\left|S\right|>4$ and $S'=1$ otherwise.  
\end{enumerate}
\end{lemma}
\begin{proof}
\begin{enumerate}
\item Let $S$ be a subloop of $Q_{n}$. By Theorem \ref {thm:culbert}, $S\leq \mathbb{C}_4$ when $\left|S\right|\leq4$; $\mathbb{C}_{4}$ is an abelian group, hence $Z\left(S\right)=S$. Let $\left|S\right|>4$. By Theorem \ref {thm:culbert}, $\left\langle 1,x\right\rangle\leq \mathbb{C}_{4}$ and $\left\langle -1,x\right\rangle\leq \mathbb{C}_{4}$, $\mathbb{C}_{4}$ is abelian and therefore $\{1,-1\}\in C(S)$. Let $x\in S\backslash\{\pm1\}$, choose an element $y\notin \{\pm1,\pm x\}$. Then $\left\langle x,y\right\rangle\cong \mathbb{H}_{8}$ by Theorem \ref {thm:culbert}, and $[x,y]=-1$. It follows that $C(S)=\{1,-1\}$. Also, $\left\langle 1,x,y\right\rangle \leq \mathbb{H}_{8}$ and $\left\langle -1,x,y\right\rangle \leq \mathbb{H}_{8}$, therefore $[1,x,y]=1$ and $[-1,x,y]=1$ for any $x,y\in S$, and $\{1,-1\}\in N(S)$. It follows that $Z\left
(S\right)=\left\{1,-1\right\}$.  
\item Let $\left|S\right|>8$. A group $S/Z(S)$ is abelian, hence $A(S)\leq Z(S)$. Also, $A(S)\neq 1$ since $S$ is not a group, so $A(S)=Z(S)$. Let $\left|S\right|\leq 8$, then $S\leq\mathbb{H}_{8}$ and $\mathbb{H}_{8}$ is a group, so $A(S)=1$.
\item Let $\left|S\right|>4$. A group $S/Z(S)$ is abelian, hence $S'\leq Z(S)$. Also, $S'\neq 1$ since $S$ is not an abelian group, so $S'= Z(S)$. Let $\left|S\right|\leq 4$, then $S\leq\mathbb{C}_{4}$ and $\mathbb{C}_{4}$ is an abelian group, so $S'=1$.
\qedhere
\end{enumerate}
\end{proof}
\begin{proposition}\label{prop:prop}
Let  $Q_{n}$ be a Cayley-Dickson loop. The following holds  
\begin{enumerate}
	\item Conjugates of the elements of $Q_{n}$ are $x^{*}=-x$ for $x\in  Q_{n}\backslash\left\{1,-1\right\}$, $1^{*}=1$, $(-1)^{*}=-1$. \label{eqn:conjugates}
	\item Orders of the elements of $Q_{n}$ are $\left|x\right|=4$ for $x\in  Q_{n}\backslash\left\{1,-1\right\}$, $\left|1\right|=1$, $\left|-1\right|=2$. \label{eqn:order}
	\item Inverses of the elements of $Q_{n}$ are $x^{-1}=x^{*}$ for all $x \in  Q_{n}$.
	\item Size of $Q_{n}$ is $2^{n+1}$.
	\item For $k\leq n$, $Q_{k}$ embeds into $Q_{n}$, $k\in \mathbb{N}$.
\end{enumerate}
\end{proposition}
\begin{proof}
\begin{enumerate}
	\item By induction on $n$. In $\mathbb{R}_{2}$, $1\cdot1=-1\cdot(-1)=1$. Suppose $x^{*}=-x$ holds for all $x\in  Q_{n} \backslash \left\{\pm1\right\}$, then in $Q_{n+1}$ by definition $(x,0)^{*}=(x^{*},0)=(-x,0)=-(x,0)$ and $(x,1)^{*}=(-x,1)=-(x,1)$. 
	\item By induction on $n$. In $\mathbb{C}_{4}$, $(1,0)(1,0)=(1,0)$ and $(1,1)(1,1)=-(1,0)$. Suppose $x^2=-1$ holds for all $x\in  Q_{n}\backslash\left\{\pm1\right\}$, then in $Q_{n+1}$  $(x,0)(x,0)=(xx,0)=(-1,0)$ and $(x,1)(x,1)=(-x^{*}x,0)=(xx,0)=(-1,0)$.
	\item Follows from \ref{eqn:conjugates}. and \ref{eqn:order}. $x^{*}x=(-x)x=-(xx)=1=-(xx)=x(-x)=xx^{*}$ when $x\neq \pm 1$ and $(\pm1)^2=1$. 
	\item By definition.
	\item $Q_{k}\cong\left\{(x,0)\left|\right.(x,0)\in Q_{k+1}\right\}$, $k\in \mathbb{N}$. 
	\qedhere
\end{enumerate}
\end{proof}
\begin{definition}\label{def:ip}
A loop $L$ is an \emph{inverse property loop} if for every $x\in L$ there is $x^{-1}\in L$ such that $x^{-1}(xy)=y=(yx)x^{-1}$ for every $y\in L$.
\end{definition}
\begin{corollary}\label{cor:ip}
Cayley-Dickson loop is an inverse property loop.
\end{corollary}
\begin{proof}
$x^{-1}=x^{*}$ by Proposition \ref{prop:prop}. $x^{*}(xy)=(x^{*}x)y=y=y(xx^{*})=(yx)x^{*}$ by Corollary~\ref{cor:diassoc}.
\end{proof}
\begin{definition}\label{def:comm}
Let L be a loop. For any $x, y \in L$ define \emph{commutator} $[x,y]$ by $xy=(yx)[x,y]$.
\end{definition}
\begin{definition}\label{def:assoc}
Let L be a loop. For any $x, y, z \in L$ define \emph{associator} $[x,y,z]$ by \\$xy\cdot z=(x\cdot yz)[x,y,z]$.
\end{definition}
\begin{theorem}[Moufang \cite{Moufang:35}]\label{thm:moufang}
Let $\left(M,\cdot \right)$ be a Moufang loop. If  $[x,y,z]=1$ for some $x,y,z\in M$, then $x,y,z$ generate a group in $\left(M,\cdot \right)$.
\end{theorem} 
\begin{lemma}\label{lemma:comm-assoc}
Let $x,y,z$ be elements of $Q_{n}.$ The following holds
\begin{enumerate}
	\item Commutator $[x,y]=-1$ when $\left\langle x,y\right\rangle\cong \mathbb{H}_{8}$ and $[x,y]=1$ when $\left\langle x,y\right\rangle<\mathbb{H}_{8}$.
	\item Associator $[x,y,z]=1$ or $[x,y,z]=-1$. In particular, $[x,y,z]=1$ when $\left\langle x,y,z\right\rangle\leq \mathbb{H}_{8}$ and $[x,y,z]=-1$ when $\left\langle x,y,z\right\rangle\cong \mathbb{O}_{16}$. 
\end{enumerate}
\end{lemma}
\begin{proof}
\begin{enumerate}
	\item By Theorem \ref{thm:culbert}, $\left\langle x,y\right\rangle< \mathbb{H}_{8}$ when either $x=\pm1$, or $y=\pm1$, or both, or $x=\pm y$, moreover, $\left\langle x,y\right\rangle< \mathbb{H}_{8}$ implies that $\left\langle x,y\right\rangle\leq \mathbb{C}_{4}$. The complex group $\mathbb{C}_{4}$ is abelian, hence $[x,y]=1$ when $\left\langle x,y\right\rangle<\mathbb{H}_{8}$. Next, suppose $\left\langle x,y\right\rangle\cong\mathbb{H}_{8}$, i.e., $x\neq\pm1$, $y\neq\pm1$, $x\neq\pm y$. The quaternion group $\mathbb{H}_{8}$ is not abelian, therefore $\left[x,y\right]=-1$.	
	\item By induction on $n$. Holds on elements of $\mathbb{R}_{2}$. Suppose $[x,y,z]=1$ or $[x,y,z]=-1$ $\forall x,y,z \in  Q_{n}$. Then in $Q_{n+1}$, $(x,x_{n+1})(y,y_{n+1})\cdot(z,z_{n+1})=(f(x,y,z), (x_{n+1}+y_{n+1}+z_{n+1})$ mod $2)$, where $x_{n+1},y_{n+1},z_{n+1}\in\left\{0,1\right\}$ and $f(x,y,z)$ is some product of $x,y,z,x^{*},y^{*},z^{*}$ and possibly $-1$. Recall that $x^{*}=x$ or $x^{*}=-x$ for $x\in Q_{n}$, therefore $f(x,y,z)$ is in fact the product of $x,y,z$, each occuring exactly once, and possibly $-1$. Similarly, $(x,x_{n+1})\cdot(y,y_{n+1})(z,z_{n+1})=(g(x,y,z), (x_{n+1}+y_{n+1}+z_{n+1})$ mod $2),$ where $g(x,y,z)$ is some product of $x,y,z$, each occuring exactly once, and possibly $-1$. In other words, $f(x,y,z)$ and $g(x,y,z)$ only differ by a sign, which shows that either 
\[ [(x,x_{n+1}),(y,y_{n+1}),(z,z_{n+1})]=1\mbox{ or }[(x,x_{n+1}),(y,y_{n+1}),(z,z_{n+1})]=-1.\]
Finally, $\mathbb{H}_{8}$ is associative, therefore $[x,y,z]=1$ when $\left\langle x,y,z\right\rangle\leq\mathbb{H}_{8}$.\\
$\mathbb{O}_{16}$ is a Moufang loop and not a group, therefore by Moufang's Theorem $[x,y,z]=-1$ when $\left\langle x,y,z\right\rangle\cong\mathbb{O}_{16}$. 
\qedhere
\end{enumerate}
\end{proof}
Let $\mathbb{Z}_{2}$ be a cyclic group of order 2.
\begin{remark}\label{cor:Zn}
A group $Q_{n}/\{1,-1\}$ is abelian and isomorphic to (multiplicative) $(\mathbb{Z}_{2})^{n}$.
\end{remark}
\begin{proof}
Follows from Lemma~\ref{lemma:center} and construction~\eqref{eqn:CD2}.
\end{proof}
\begin{lemma}\label{lemma:order}
Let $B$ be a subloop of $Q_{n}$. The following holds
\begin{enumerate}
	\item If $B\neq1$ and $x\in  Q_{n}\backslash B$, then $\left|\left\langle B,x\right\rangle\right|=2\left|B\right|$. \label{lemma:order1}
	\item If $B=1$ and $x\in  Q_{n}\backslash B$, then $\left\langle B,x\right\rangle=\left\{1,-1,x,-x\right\}$.
	\item Any $n$ elements of a Cayley-Dickson loop generate a subloop of size $2^k,$ $k\leq n+1$. \label{lemma:order3}
	\item The size of $B$ is $2^{m}$ for some $m\leq n$.
\end{enumerate}
\end{lemma}
\begin{proof} 
\begin{enumerate}
	\item Let $1\neq B\leq  Q_{n}$ and $x\in  Q_{n}\backslash B$. By Lemma \ref{lemma:center}, $Z(Q_{n})\leq B$ and $Z(Q_{n})\leq \left\langle B,x\right\rangle$, then $B/Z(Q_{n})$ and $\left\langle B,x\right\rangle/Z(Q_{n})$ are subgroups of $Q_{n}/Z(Q_{n})\cong(\mathbb{Z}_{2})^n$. It follows that $\left|\left\langle B,x\right\rangle/Z(Q_{n})\right|=2\left|B/Z(Q_{n})\right|$ because we work in the vector space $(\mathbb{Z}_{2})^n$ and we added another vector.
	\item Let $B=1$. If $x\neq-1$ then $x^2=-1$ by Proposition \ref{prop:prop} and $\left\langle B,x\right\rangle=\left\langle x\right\rangle=\left\{1,-1,x,-x\right\}$. Also, $\left\langle B,-1\right\rangle=\left\{1,-1\right\}$.  
	\item By induction on n. The size of $\left\langle x\right\rangle$ is $1,2$ or $4$. Suppose $n$ elements of a Cayley-Dickson loop generate a subloop $B$ of size $2^{k}$ for some $k\leq n+1$. Add an element $x$ to $B$. If $x\in B$, then $\left|\left\langle B,x\right\rangle\right|=\left|B\right|=2^{k}$, $k\leq n+1\leq n+2$. If $x\notin B$, then $\left|\left\langle B,x\right\rangle\right|=2\left|B\right|=2^{k+1}$, $k+1\leq n+2$, by \ref{lemma:order1}.
	\item Follows from \ref{lemma:order3}.
	\qedhere
\end{enumerate}
\end{proof}
\section{Cayley-Dickson loops are Hamiltonian}\label{sec:hamiltonian}
We show that the Cayley-Dickson loops are Hamiltonian. Norton \cite{Norton:52} formulated a number of theorems characterizing diassociative Hamiltonian loops and showed that the octonion loop is Hamiltonian, however, at that time he did not study the generalized Cayley-Dickson loops. It is showed computationally in~\cite{Cawagas:09} that $\mathbb{T}_{64}$ is Hamiltonian. 
\begin{definition}\label{def:hamiltonian}
A \emph{Hamiltonian loop} is a loop in which every subloop is normal.
\end{definition}
\begin{theorem}\label{thm:hamiltonian}
Cayley-Dickson loop $Q_{n}$ is Hamiltonian.
\end{theorem}
\begin{proof} Let $S$ be a subloop of $Q_{n}$, $s\in S$, $x,y\in  Q_{n}$. Using Lemma \ref{lemma:comm-assoc} and Lemma \ref{lemma:center}, 
\begin{eqnarray}
	\nonumber xs & = & [x,s]sx\in\left\{sx,-sx\right\}\subseteq Sx, \\
	\nonumber (xs)y & = & [x,s,y]x(sy)\in\left\{x(sy),-x(sy)\right\}\subseteq x(Sy), \\ 
	\nonumber x(ys) & = & [x,y,s](xy)s\in\left\{(xy)s,-(xy)s\right\}\subseteq (xy)S.
\qedhere
\end{eqnarray}
\end{proof}
\begin{theorem}\textbf{(Norton)}\label{thm:norton2}
If $A$ is an abelian group with elements of odd order, $T$ is an abelian group with exponent 2, and $K$ is a diassociative loop such that
		\begin{enumerate}
			\item elements of K have order 1, 2 or 4,
			\item there exist elements $x,y$ in $K$ such that $\left\langle x,y\right\rangle\cong \mathbb{H}_{8}$,
			\item every element of K of order 2 is in the center,
			\item if $x,y,z\in K$ are of order 4, then $x^2=y^2=z^2$, \\$xy=d\cdot yx$ where $d=1$ or $d=x^2$, \\and $xy\cdot z=h(x\cdot yz)$ where $h=1$ or $h=x^2$,  
		\end{enumerate}
then their direct product $A\times T\times K$ is a diassociative Hamiltonian loop.
\end{theorem}
Theorem~\ref{thm:norton2} with $A=T=1$ can alternatively be used to establish the result for all Cayley-Dickson loops. 
\section{Automorphism groups of the Cayley-Dickson loops}\label{sec:automorphism}
In this section we study the automorphism groups of the Cayley-Dickson loops. 
\begin{definition}\label{def:autom}
Let L be a loop. A map $\phi:L\mapsto L$ is an \emph{automorphism} if it is a bijective homomorphism. 
\end{definition}
\begin{definition}\label{def:automgrp}
The set of all automorphisms of a loop L forms a group under composition, called the \emph{automorphism group} and denoted by $Aut(L)$. 
\end{definition}
\begin{definition}\label{def:orbit}
Define the \emph{orbit} of a set $X$ under the action of a group $G$ by $O_{G}(X)=\left\{gx\ |\ g\in G,\ x\in X\right\}$.
\end{definition}
\begin{definition}\label{def:stabilizer}
Define the (pointwise) \emph{stabilizer} of a set $X$ in $G$ by $G_{X}=\left\{g\in G\ |\ gx=x,\ x\in X\right\}$.
\end{definition}
\begin{theorem}[Orbit-Stabilizer Theorem \cite{Rotman:10} p.67]\label{thm:orbstab}
Let $G$ be a finite group acting on a finite set $X$, then $|O_{G}(X)|=\left[G:G_{X}\right]=\frac{\left|G\right|}{\left|G_{X}\right|}$.
\end{theorem}
We use Theorem \ref{thm:orbstab} to find an upper bound on the size of $Aut(\mathbb{C}_{4})$ and $Aut(\mathbb{H}_{8})$. Consider $G=Aut(\mathbb{C}_{4})$. Any automorphism on $G$ fixes $1$ and $-1$, therefore it is only possible for an automorphism to map $i_{1}\mapsto i_{1}$ (e.g., the identity map), and $i_{1}\mapsto -i_{1}$ (e.g., conjugation). The size of the orbit $O_{G}(i_{1})$ is therefore 2. Notice that $G_{\left\{i_{1}\right\}}=G_{\mathbb{C}_{4}}$, since $\mathbb{C}_{4}$ is generated by $i_{1}$. It follows that
\[
\left|G\right| =  \left|O_{G}(i_{1})\right|\cdot\left|G_{\left\{i_{1}\right\}}\right|=\left|O_{G}(i_{1})\right|=2.
\]
Next, let $G=Aut(\mathbb{H}_{8})$. Again, $1$ and $-1$ are fixed by any automorphism and are not in $O_{G}(i_{1})$, therefore the size of $|O_{G}(i_{1})|$ can be at most $|\mathbb{H}_{8}|-2=6$. When $i_{1}$ is stabilized, $\left|G_{\left\{i_{1}\right\}}\right|=\left|O_{G_{\left\{i_{1}\right\}}}(i_{2})\right|\cdot\left|G_{\left\{i_{1},i_{2}\right\}}\right|$, moreover, $G_{\left\{i_{1},i_{2}\right\}}=G_{\mathbb{H}_{8}}$, since $\mathbb{H}_{8}$ is generated by $\{i_{1},i_{2}\}$. The orbit $O_{G_{\left\{i_{1}\right\}}}(i_{2})$ can have the size at most $|\mathbb{H}_{8}|-4=4$, because the set $\{1,-1,i_{1},-i_{1}\}$ is fixed. We have
\begin{equation}\label{eqn:autsize}
\left|G\right| =  \left|O_{G}(i_{1})\right|\cdot\left|G_{\left\{i_{1}\right\}}\right| = \left|O_{G}(i_{1})\right|\cdot\left|O_{G_{\left\{i_{1}\right\}}}(i_{2})\right|\cdot\left|G_{\left\{i_{1},i_{2}\right\}}\right| = \left|O_{G}(i_{1})\right|\cdot\left|O_{G_{\left\{i_{1}\right\}}}(i_{2})\right|\leq 6\cdot 4=24. 
\end{equation}
It has been shown, in fact, (see, e.g., \cite{Zassenhaus:99} p.148), that $Aut(\mathbb{H}_{8})$ is isomorphic to the symmetric group $S_{4}$ of size 24. \\
It has been established in~\cite{KocaKoc:1995} that $Aut(\mathbb{O}_{16})$ has size 1344 and is an extension of the elementary abelian group $(\mathbb{Z}_{2})^3$ of order~8 by the simple group $PSL_{2}(7)$ of order~168. One can use the approach similar to \eqref{eqn:autsize} to see what $Aut(\mathbb{O}_{16})$ looks like. \\
To get an idea about the general case, we calculated the automorphism groups of $\mathbb{S}_{32}$ and $\mathbb{T}_{64}$ using LOOPS package for GAP~\cite{NagyVojt:06}. Summarizing, the sizes of the automorphism groups of the first five Cayley-Dickson loops~are
\begin{eqnarray}
\nonumber \left|\mbox{Aut}\left(\mathbb{C}_{4}\right)\right| & = & 2,\\
\nonumber \left|\mbox{Aut}\left(\mathbb{H}_{8}\right)\right| & = & 24=6\cdot4,\\
\nonumber \left|\mbox{Aut}\left(\mathbb{O}_{16}\right)\right| & = & 1344=14\cdot12\cdot8,\\
\nonumber \left|\mbox{Aut}\left(\mathbb{S}_{32}\right)\right| & = & 2688=2\cdot(14\cdot12\cdot8),\\
\nonumber \left|\mbox{Aut}\left(\mathbb{T}_{64}\right)\right| & = & 5376=2\cdot2\cdot(14\cdot12\cdot8).
\end{eqnarray}
One may notice that the automorphism groups of $\mathbb{C}_{4}$, $\mathbb{H}_{8}$ and $\mathbb{O}_{16}$ are as big as they possibly can be, subject to the obvious structural restrictions in $\mathbb{C}_{4},\mathbb{H}_{8},\mathbb{O}_{16}$, only fixing $\{1,-1\}$ ($1$ is the only element of order $1$, and $-1$ is the only element of order $2$). On the contrary, the automorphism groups of $\mathbb{S}_{32}$ and $\mathbb{T}_{64}$ are only double the size of the preceeding ones. Theorem~\ref{thm:autom} below explains such behavior. We denote $e=(1_{Q_{n-1}},1)\in  Q_{n}$ and use it further in the text.
\begin{theorem}\label{thm:autom}
Let $n\geq 4$. If $\phi: Q_{n}\mapsto  Q_{n}$ is an automorphism and $\psi=\phi\upharpoonright_{Q_{n-1}}$, then 
\begin{enumerate}
	\item $\phi\left(1\right)=1$, $\phi\left(-1\right)=-1$,
	\item $\phi\left(e\right)=e$ or $\phi\left(e\right)=-e$,
	\item $\psi\in Aut(Q_{n-1})$,
	\item $\phi((x,1))=\psi(x)\phi(e),\forall x \in  Q_{n-1}$.
\end{enumerate} 
\end{theorem}
We establish several auxiliary results and use them to prove Theorem \ref{thm:autom} at the end of the chapter. The following lemma shows that all subloops of $Q_{n}$ of size $16$ fall into two isomorphism classes. In particular, any such subloop is either isomorphic to $\mathbb{O}_{16}$, the octonion loop, or $\tilde{\mathbb{O}}_{16}$, the quasioctonion loop, described in~\cite{Cawagas:04,Culbert:07}. The octonion loop is Moufang, however, the quasioctonion loop is not. We take $\left\langle i_{1},i_{2},i_{3}\right\rangle=\pm\{1,i_{1},i_{2},i_{1}i_{2},i_{3},i_{1}i_{3},i_{2}i_{3},i_{1}i_{2}i_{3}\}$ as a canonical octonion loop, and 
$\left\langle i_{1},i_{2},i_{3}i_{4}\right\rangle=\pm\{1,i_{1},i_{2},i_{1}i_{2},i_{3}i_{4},i_{1}i_{3}i_{4},i_{2}i_{3}i_{4},i_{1}i_{2}i_{3}i_{4} \}$ as a canonical quasioctonion loop in $\mathbb{S}_{32}$. 
We use LOOPS package for GAP~\cite{NagyVojt:06} in Lemma \ref{lemma:quasioct} and further in the text to establish the isomorphisms between the subloops we construct, and either $\mathbb{O}_{16}$ or $\tilde{\mathbb{O}}_{16}$.	
\begin{lemma}\label{lemma:quasioct}
If $x,y,z$ are elements of $Q_{n}$ such that $\left|\left\langle x,y,z\right\rangle\right|=16$, then either
\begin{equation}
\nonumber \left\langle x,y,z\right\rangle \cong \mathbb{O}_{16} \mbox{ or } \left\langle x,y,z\right\rangle \cong \tilde{\mathbb{O}}_{16}.
\end{equation}
\end{lemma}
\begin{proof}
Let $x,y,z \in Q_{n}$ such that $\left|\left\langle x,y,z\right\rangle\right|=16$. We want to construct a loop 
\[
\left\langle x,y,z\right\rangle=\pm\{1,x,y,xy,z,xz,yz,(xy)z\}.
\]
Fix the associators $[x,y,z]$, $[x,z,y]$, and $[x,y,xz]$. Using diassociativity and Lemma~\ref{lemma:comm-assoc}.1,
\begin{eqnarray}
\label{eq:quazioct0} x((xy)z) & = & [x,y,z]x(x(yz))=[x,y,z](xx)(yz)=-[x,y,z]yz,\\
y(xz) & = & -(xz)y=-[x,z,y]x(zy)=[x,z,y]x(yz)=[x,y,z][x,z,y](xy)z,\\
\nonumber y((xy)z) & = & -((xy)z)y= -[x,y,z](x(yz))y= [x,y,z](x(zy))y\\
\nonumber & = &[x,y,z][x,z,y]((xz)y)y= [x,y,z][x,z,y](xz)(yy)\\
& = & -[x,y,z][x,z,y](xz),\\
\nonumber(xz)((xy)z) & = & -((xy)z)(xz)=-[x,y,z](x(yz))(xz)=[x,y,z](x(zy))(xz)\\
\nonumber & = & [x,y,z][x,z,y]((xz)y)(xz)=-[x,y,z][x,z,y](y(xz))(xz)\\
& = & -[x,y,z][x,z,y]y((xz)(xz))=[x,y,z][x,z,y]y,\\
\nonumber (yz)((xy)z) & = & [x,y,z](yz)(x(yz))=-[x,y,z](x(yz))(yz)\\
& = & -[x,y,z]x((yz)(yz))=[x,y,z]x,\\
\nonumber (xy)(xz)&=&[x,y,xz]x(y(xz))=-[x,y,xz]x((xz)y)=-[x,z,y][x,y,xz]x(x(zy))\\
\nonumber & = & -[x,z,y][x,y,xz](xx)(zy)=[x,z,y][x,y,xz](zy)\\
\label{eq:quazioct1} & = & -[x,z,y][x,y,xz](yz).
\end{eqnarray}
Multiplying \eqref{eq:quazioct1} by $(xy)$ on the left, 
\begin{equation}
(xy)(yz)=[x,z,y][x,y,xz]xz.
\end{equation}
Multiplying \eqref{eq:quazioct1} by $(xz)$ on the right, 
\begin{equation}
\label{eq:quazioct2} (yz)(xz)=[x,z,y][x,y,xz]xy.
\end{equation}
Equalities \eqref{eq:quazioct0}-\eqref{eq:quazioct2} together with some trivial calculations result in Table \ref{tab:MT1}, i.e., it is sufficient to fix $[x,y,z]$, $[x,z,y]$ and $[x,y,xz]$ in order to uniquely define $\left\langle x,y,z\right\rangle$. 
\begin{table}[h]\scriptsize
	\centering
		\begin{tabular}{c||c|c|c|c|c|c|c}			
			1    &  x           & y                   &  xy               & z   &  xz                &  yz                & (xy)z\\ \hline \hline
			x    & -1           & xy                  &  -y               & xz  &  -z                &[x,y,z](xy)z       &-[x,y,z]yz  \\ \hline
			y    & -xy          & -1                  &   x               &  yz &[x,y,z][x,z,y](xy)z &     -z             &-[x,y,z][x,z,y]xz \\ \hline
			xy   &  y           &  -x                 &  -1               &(xy)z&-[x,z,y][x,y,xz]yz&[x,z,y][x,y,xz]xz & -z \\ \hline
			z    &   -xz        &  -yz                &-(xy)z             & -1  &  x                 &         y          &  xy\\ \hline
			xz   &   z          &-[x,y,z][x,z,y](xy)z&[x,z,y][x,y,xz]yz&  -x & -1                 &-[x,z,y][x,y,xz]xy&[x,y,z][x,z,y]y\\ \hline
			yz   &-[x,y,z](xy)z&  z                   &-[x,z,y][x,y,xz]xz&  -y &[x,z,y][x,y,xz]xy& -1                  &[x,y,z]x\\ \hline
		  (xy)z&[x,y,z]yz     &[x,y,z][x,z,y]xz    &   z               & -xy &-[x,y,z][x,z,y]y   &-[x,y,z]x           & -1
		\end{tabular}
	\caption{Multiplication table of $\left\langle x,y,z\right\rangle$}
	\label{tab:MT1}
\end{table}
We need to consider the following cases:\\
If $[x,y,z]=[x,z,y]=[x,y,xz]=-1$, then $\left\langle x,y,z\right\rangle \cong \mathbb{O}_{16}$ by $\{x,y,z\}\mapsto\{i_{1},i_{2},i_{3}\}$.\\
If $[x,y,z]=[x,z,y]=-1$, $[x,y,xz]=1$, then $\left\langle x,y,z\right\rangle\cong \tilde{\mathbb{O}}_{16}$ by $\{xz, yz, z\}\mapsto\{i_{1},i_{2},i_{3}i_{4}\}$.\\
If $[x,y,z]=[x,y,xz]=-1$, $[x,z,y]=1$, then $\left\langle x,y,z\right\rangle\cong \tilde{\mathbb{O}}_{16}$ by $\{x,z,y\}\mapsto\{i_{1},i_{2},i_{3}i_{4}\}$.\\
If $[x,y,z]=-1$, $[x,z,y]=[x,y,xz]=1$, then $\left\langle x,y,z\right\rangle\cong \tilde{\mathbb{O}}_{16}$ by $\{y,-xz,x\}\mapsto\{i_{1},i_{2},i_{3}i_{4}\}$.\\
If $[x,y,z]=1$, $[x,z,y]=[x,y,xz]=-1$, then $\left\langle x,y,z\right\rangle\cong \tilde{\mathbb{O}}_{16}$ by $\{-xy, z, x\}\mapsto\{i_{1},i_{2},i_{3}i_{4}\}$.\\
If $[x,y,z]=[x,y,xz]=1$, $[x,z,y]=-1$, then $\left\langle x,y,z\right\rangle\cong \tilde{\mathbb{O}}_{16}$ by $\{x,y,z\}\mapsto\{i_{1},i_{2},i_{3}i_{4}\}$.\\
If $[x,y,z]=[x,z,y]=1$, $[x,y,xz]=-1$, then $\left\langle x,y,z\right\rangle\cong \tilde{\mathbb{O}}_{16}$ by $\{y,z,x\}\mapsto\{i_{1},i_{2},i_{3}i_{4}\}$.\\
If $[x,y,z]=[x,z,y]=[x,y,xz]=1$, then $\left\langle x,y,z\right\rangle\cong \tilde{\mathbb{O}}_{16}$ by $\{x,-yz,y\}\mapsto\{i_{1},i_{2},i_{3}i_{4}\}$.\qedhere
\end{proof}
Next, we study the associators in $Q_n$. We use the result to prove Lemmas \ref{lemma:aut2} and \ref{lemma:aut1}. 
\begin{lemma} \label{rmk:assoc} Let $x,y,z\in Q_{n-1},$ then in $Q_n$
\begin{enumerate}[(a)]
	\item $[(x,0),(y,0),(z,1)]=[x,y][z,y,x],$	
	\item $[(x,0),(y,1),(z,0)]=[x,z][y,x,z][y,z,x],$	
	\item $[(x,0),(y,1),(z,1)]=[x,y][x,z][z,x,y][x,z,y],$
  \item $[(x,1),(y,0),(z,0)]=[y,z][x,y,z],$
	\item $[(x,1),(y,0),(z,1)]=[y,x][y,z][z,y,x],$
	\item $[(x,1),(y,1),(z,0)]=[z,x][z,y][y,x,z][y,z,x],$
	\item $[(x,1),(y,1),(z,1)]=[x,y][x,z][y,z][z,x,y][x,z,y].$
\end{enumerate}\end{lemma}
\begin{proof}
\begin{enumerate}[(a)]
	\item $(x,0)(y,0)\cdot(z,1)=(xy,0)(z,1)=(z\cdot xy,1)=[x,y](z\cdot yx,1)\\
	=[x,y][z,y,x](zy\cdot x,1)=[x,y][z,y,x]((x,0)(zy,1))=[x,y][z,y,x]((x,0)\cdot(y,0)(z,1)).$	
	\item $(x,0)(y,1)\cdot(z,0)=(yx,1)(z,0)=(yx\cdot z^{*},1)=[y,x,z](y\cdot xz^{*},1)=[x,z][y,x,z](y\cdot z^{*}x,1)\\=[x,z][y,x,z][y,z,x](yz^{*}\cdot x,1)=[x,z][y,x,z][y,z,x]((x,0)(yz^{*},1))\\
	=[x,z][y,x,z][y,z,x]((x,0)\cdot (y,1)(z,0)).$
	\item $(x,0)(y,1)\cdot (z,1)=(yx,1)(z,1)=(-z^{*}\cdot yx,0)=[x,y](-z^{*}\cdot xy,0)\\
	=[x,y][z,x,y](-z^{*}x\cdot y,0)=[x,y][x,z][z,x,y](x(-z^{*})\cdot y,0)\\
	=[x,y][x,z][z,x,y][x,z,y](x\cdot(-z^{*})y,0)=[x,y][x,z][z,x,y][x,z,y]((x,0)\cdot(-z^{*}y,0))\\
	=[x,y][x,z][z,x,y][x,z,y]((x,0)\cdot(y,1)(z,1)).$
	\item $(x,1)(y,0)\cdot (z,0)=(xy^{*},1)(z,0)=(xy^{*}\cdot z^{*},1)=[x,y,z](x\cdot y^{*}z^{*},1)\\
	=[x,y,z]((x,1)((y^{*}z^{*})^{*},0))=[x,y,z]((x,1)(zy,0))=[y,z][x,y,z]((x,1)(yz,0))\\
	=[y,z][x,y,z]((x,1)\cdot(y,0)(z,0)).$
	\item $(x,1)(y,0)\cdot (z,1)=(xy^{*},1)(z,1)=(-z^{*}\cdot xy^{*},0)=[y,x](-z^{*}\cdot y^{*}x,0)\\
	=[y,x][z,y,x](-z^{*}y^{*}\cdot x,0)=[y,x][z,y,x]((x,1)(-(-z^{*}y^{*})^{*},1))\\
	=[y,x][z,y,x]((x,1)(yz,1))=[y,x][y,z][z,y,x]((x,1)(zy,1))\\
	=[y,x][y,z][z,y,x]((x,1)\cdot(y,0)(z,1)).$
	\item $(x,1)(y,1)\cdot (z,0)=(-y^{*}x,0)(z,0)=(-y^{*}x\cdot z,0)=[y,x,z](-y^{*}\cdot xz,0)\\
	=[z,x][y,x,z](-y^{*}\cdot zx,0)=[z,x][y,x,z][y,z,x](-y^{*}z\cdot x,0)\\
=[z,x][y,x,z][y,z,x]((x,1)(-(-y^{*}z)^{*},1))=[z,x][y,x,z][y,z,x]((x,1)(z^{*}y,1))\\
=[z,x][z,y][y,x,z][y,z,x]((x,1)(yz^{*},1))=[z,x][z,y][y,x,z][y,z,x]((x,1)\cdot (y,1)(z,0)).$
	\item $(x,1)(y,1)\cdot(z,1)=(-y^{*}x,0)(z,1)=(z\cdot (-y^{*})x,1)=[x,y](z\cdot x(-y^{*}),1)\\
	=[x,y][z,x,y](zx\cdot(-y^{*}),1)=[x,y][x,z][z,x,y](xz\cdot(-y^{*}),1)\\
	=[x,y][x,z][z,x,y][x,z,y](x\cdot z(-y^{*}),1)=[x,y][x,z][z,x,y][x,z,y]((x,1)((z(-y^{*}))^{*},0))\\
	=[x,y][x,z][z,x,y][x,z,y]((x,1)(-yz^{*},0))=[x,y][x,z][y,z][z,x,y][x,z,y]((x,1)(-z^{*}y,0))\\
	=[x,y][x,z][y,z][z,x,y][x,z,y]((x,1)\cdot (y,1)(z,1)).$
	\qedhere	
\end{enumerate}
\end{proof}
Lemma \ref{lemma:aut2} shows that $e\in Q_{n}$ is special; if we consider a subloop $\left\langle x,y,e\right\rangle$ of $Q_{n}$ such that $\left|\left\langle x,y,e\right\rangle\right|=16$, then $\left\langle x,y,e\right\rangle$ is always a copy of the octonion loop $\mathbb{O}_{16}$. Lemma \ref{lemma:aut7} shows that this, however, is not the case for any element of $Q_{n}\backslash\{\pm e\}$. Therefore, an automorphism on $Q_{n}$ cannot map $e$ to an element $x\in Q_{n}\backslash\{\pm e\}$. Also, we use Lemma \ref{lemma:aut3} to show that an element $(x,0)$ of $Q_{n}$ is contained in more copies of $Q_{n-1}$ than an element $(y,1)$, and hence an automorphism on $Q_{n}$ cannot map $(x,0)$ to $(y,1)$ for any $x,y\in Q_{n-1}$.
\begin{lemma}\label{lemma:aut2}
$\left\langle x,y,e\right\rangle\cong\mathbb{O}_{16}$ for any $x,y\in  Q_{n}$ such that $e\notin \left\langle x,y\right\rangle\cong \mathbb{H}_{8}$.
\end{lemma}
\begin{proof}
Let $x,y$ be elements of $Q_{n}$ such that $e\notin \left\langle x,y\right\rangle\cong \mathbb{H}_{8}$. As follows from the proof of Lemma~\ref{lemma:quasioct}, in order to prove that $\left\langle x,y,e\right\rangle\cong\mathbb{O}_{16}$, it is sufficient to show that  
\begin{equation}
[x,y,e]=[x,e,y]=[x,y,xe]=-1.
\end{equation}
Let $\overline{x},\overline{y}$ be elements of $Q_{n-1}$. We use Lemma~\ref{rmk:assoc}, and consider the following cases:\\
If $x=(\overline{x},0),y=(\overline{y},0)$, then $xe=(\overline{x},0)(1,1)=(\overline{x},1)$, and
	\begin{eqnarray}
\nonumber [x,y,e] & = & [(\overline{x},0),(\overline{y},0),(1,1)]=[\overline{x},\overline{y}][1,\overline{y},\overline{x}]=-1,\\ 	     
\nonumber [x,e,y] & = & [(\overline{x},0),(1,1),(\overline{y},0)]=[\overline{x},\overline{y}][1,\overline{x},\overline{y}][1,\overline{y},\overline{x}]=-1,\\
\nonumber [x,y,xe] & = & [(\overline{x},0),(\overline{y},0),(\overline{x},1)]=[\overline{x},\overline{y}][\overline{x},\overline{y},\overline{x}]=-1.
	\end{eqnarray}
If $x=(\overline{x},0),y=(\overline{y},1)$, then $xe=(\overline{x},0)(1,1)=(\overline{x},1)$, and 
	\begin{eqnarray}
\nonumber [x,y,e] & = & [(\overline{x},0),(\overline{y},1),(1,1)]=[\overline{x},\overline{y}][\overline{x},1][1,\overline{x},\overline{y}][\overline{x},1,\overline{y}]=-1,\\
\nonumber [x,e,y] & = & [(\overline{x},0),(1,1),(\overline{y},1)]=[\overline{x},1][\overline{x},\overline{y}][\overline{y},\overline{x},1][\overline{x},\overline{y},1]=-1,\\ \nonumber [x,y,xe] & = & [(\overline{x},0),(\overline{y},1),(\overline{x},1)]= [\overline{x},\overline{y}][\overline{x},\overline{x}][\overline{x},\overline{x},\overline{y}][\overline{x},\overline{x},\overline{y}]=-1.
	\end{eqnarray}
If $x=(\overline{x},1),y=(\overline{y},0)$, then $xe=(\overline{x},1)(1,1)=(-\overline{x},0)$, and
	\begin{eqnarray}
\nonumber [x,y,e] & = & [(\overline{x},1),(\overline{y},0),(1,1)]=[\overline{y},\overline{x}][\overline{y},1][1,\overline{y},\overline{x}]=-1,\\ 			\nonumber [x,e,y] & = & [(\overline{x},1),(1,1),(\overline{y},0)]=[\overline{y},\overline{x}][\overline{y},1][1,\overline{x},\overline{y}][1,\overline{y},\overline{x}]=-1,\\ \nonumber [x,y,xe] & = & [(\overline{x},1),(\overline{y},0),(-\overline{x},0)]=[\overline{y},-\overline{x}][\overline{x},\overline{y},-\overline{x}]=-1.
	\end{eqnarray}
If $x=(\overline{x},1),y=(\overline{y},1)$, then $xe=(\overline{x},1)(1,1)=(-\overline{x},0)$, and 
	\begin{eqnarray}
\nonumber [x,y,e] & = & [(\overline{x},1),(\overline{y},1),(1,1)]=[\overline{x},\overline{y}][\overline{x},1][\overline{y},1][1,\overline{x},\overline{y}][\overline{x},1,\overline{y}]=-1,\\
\nonumber [x,e,y] & = & [(\overline{x},1),(1,1),(\overline{y},1)]=[\overline{x},1][\overline{x},\overline{y}][1,\overline{y}][\overline{y},\overline{x},1][\overline{x},\overline{y},1]=-1,\\
\nonumber [x,y,xe] & = & [(\overline{x},1),(\overline{y},1),(-\overline{x},0)]=[-\overline{x},\overline{x}][-\overline{x},\overline{y}][\overline{y},\overline{x},-\overline{x}][\overline{y},-\overline{x},\overline{x}]=-1.
	\end{eqnarray}
We conclude that $[x,y,e]=[x,e,y]=[x,y,xe]=-1$ for any $x,y\in  Q_{n}$ such that $e\notin \left\langle x,y\right\rangle\cong \mathbb{H}_{8}$.  
By Lemma~\ref{lemma:quasioct}, $\left\langle x,y,e\right\rangle\cong \mathbb{O}_{16}$ by $\{x,y,e\}\mapsto\{i_{1},i_{2},i_{3}\}$. \qedhere
\end{proof}
The following lemma helps to distinguish between some copies of $\mathbb{O}_{16}$ and $\tilde{\mathbb{O}}_{16}$, and is used to prove Lemmas \ref{lemma:aut3} and \ref{lemma:aut7}.
\begin{lemma}\label{lemma:aut1}
Let $x,y,z\in  Q_{n-1},$ $n\geq 4$ be such that $\left\langle x,y,z\right\rangle\cong \mathbb{O}_{16}$. Then in $Q_{n}$ 
\begin{eqnarray}
\nonumber \left\langle (x,0),(y,0),(z,0)\right\rangle & \cong & \left\langle (x,1),(y,1),(z,1)\right\rangle\cong \mathbb{O}_{16},\\
\nonumber \left\langle (x,0),(y,0),(z,1)\right\rangle & \cong & \left\langle (x,0),(y,1),(z,1)\right\rangle\cong \tilde{\mathbb{O}}_{16}.
\end{eqnarray}
\end{lemma}
\begin{proof}
Let $x,y,z\in  Q_{n-1}$ be such that $\left\langle x,y,z\right\rangle\cong \mathbb{O}_{16}$. By Lemma \ref{lemma:comm-assoc}, $[x,y,z]=[x,z,y]=[y,x,z]=-1$, and $[x,y]=[y,z]=[x,z]=-1$. Using Lemma~\ref{rmk:assoc}, 
\begin{equation} \label{eqn:oct1}
[(x,0),(z,1),(y,0)]=[x,y][z,x,y][z,y,x]=-1
\end{equation}
shows that $\left\langle (x,0),(y,0),(z,1)\right\rangle >  \mathbb{H}_{8}$ and hence $\left|\left\langle (x,0),(y,0),(z,1)\right\rangle\right|=16$, while 
\begin{equation} \label{eqn:oct2}
[(x,0),(y,0),(z,1)]=[x,y][z,y,x]=1
\end{equation}
shows that $\left\langle (x,0),(y,0),(z,1)\right\rangle$ is not Moufang and therefore $\left\langle (x,0),(y,0),(z,1)\right\rangle\cong  \tilde{\mathbb{O}}_{16}$. Similarly, using Lemma~\ref{rmk:assoc}, 
\begin{eqnarray}
\label{eqn:oct3} [(y,1),(x,0),(z,1)] & = & [x,y][x,z][z,x,y]=-1,  \\
\label{eqn:oct4} [(x,0),(y,1),(z,1)] & = & [x,y][x,z][z,x,y][x,z,y]=1 
\end{eqnarray}
shows that $\left\langle (x,0),(y,1),(z,1)\right\rangle\cong \tilde{\mathbb{O}}_{16}$.\\
A loop $\left\langle (x,0),(y,0),(z,0)\right\rangle\cong\mathbb{O}_{16}$ as a copy of $\left\langle x,y,z\right\rangle$ in $Q_{n}$. \\
A loop $\left\langle (x,1),(y,1),(z,1)\right\rangle\cong\mathbb{O}_{16}$ by $\{(x,1),(y,1),(z,1)\} \mapsto \{i_{1},i_{2},i_{3}\}$.
\qedhere
\end{proof}
\begin{definition}
Let $B$ be a subloop of $Q_{n}$ of index 2 and $D$ be a subloop of $Q_{n-1}$ of index 2. We call $B$ a subloop of the first type when $B= Q_{n-1}$, a subloop of the second type when $B=D\oplus De$, a subloop of the third type when $B=D\oplus\left(Q_{n-1}\backslash D\right)e$.
\end{definition}
Figure \ref{fig:SubInd2inSed} illustrates all subloops of index 2 of the sedenion loop $\mathbb{S}_{32}$. Rows in the figure correspond to the subloops, columns show the elements these subloops contain. One may notice that each of the subloops is of one of three types. The following lemma shows that this is the case for all Cayley-Dickson loops. 
\begin{lemma}\label{lemma:aut4}
If $B$ is a subloop of $Q_{n}$ of index 2, then $B$ is a subloop of either the first, or the second, or the third type.
\end{lemma}
\begin{proof} 
	By Proposition \ref{prop:prop}, $Q_{n-1}$ is a subloop of $Q_{n}$ of index 2, it is of the first type. Let $B$ be a subloop of  $Q_{n}$ of index 2, we assume $B\neq  Q_{n-1}$ further in the proof. By Lemma \ref{lemma:center}, $Z(Q_{n})=\left\{1,-1\right\}\in B$. Consider $B/Z(Q_{n})$ and $Q_{n}/Z(Q_{n})$. By Remark~\ref{cor:Zn}, $Q_{n}/Z(Q_{n})\cong(\mathbb{Z}_{2})^n$. Also, there is $(a_{1},\ldots,a_{n})\in B/Z(Q_{n})$ such that $a_{n}=1$, because $B\neq  Q_{n-1}$. Define a map $\phi:B/Z(Q_{n})\mapsto B/Z(Q_{n})$ by $(x_{1},\ldots,x_{n})\mapsto (x_{1},\ldots,x_{n})(a_{1},\ldots,a_{n})=(y_{1},\ldots,y_{n})$, then $\phi$ maps elements with $x_{n}=1$ $(x_{n}=0)$ to elements with $y_{n}=0$ $(y_{n}=1)$. Hence $B/Z(Q_{n})$ contains the same number of elements that end in 0 and that end in 1, and hence a group $\{(x_{1},\ldots,x_{n})\left|\right.(x_{1},\ldots,x_{n})\in B,x_{n}=0\}$ is a subgroup of $B/Z(Q_{n})$ of index 2. This implies that $D=\{\pm(x_{1},\ldots,x_{n})\left|\right.(x_{1},\ldots,x_{n})\in B,x_{n}=0\}$ is a subloop of $B$ of index 2 and hence $D$ is a subloop of $Q_{n-1}$ of index 2. \\
Suppose $e\in B$. Define $\psi:B/Z(Q_{n})\mapsto B/Z(Q_{n})$ by $(x_{1},\ldots,x_{n-1},x_{n})\mapsto (x_{1},\ldots,x_{n-1},x_{n})(1,1)=(x_{1},\ldots,x_{n-1},y_{n})$. $\psi$ fixes coordinates $x_{1},\ldots,x_{n-1}$ and maps $x_{n}=1$ $(x_{n}=0)$ to $y_{n}=0$ $(y_{n}=1)$. Therefore when $e\in B$, we see that $B=D\oplus De$, the subloop of the second type.\\
Now suppose $e\notin B$. Suppose there is an element $(x,0)\in B/Z(Q_{n})$ such that $(x,0)(1,1)\in B/Z(Q_{n})$. By diassociativity, $(x,0)((x,0)(1,1))=((x,0)(x,0))(1,1)=(1,1)=e \in B/Z(Q_{n})$, contradicts the assumption that $e\notin B$. This means that $e\notin B$ implies $B=D\oplus (Q_{n-1}\backslash D)e$, the subloop of the third type. 	
\end{proof}
\begin{figure}
	\centering
		\includegraphics[width=14cm]{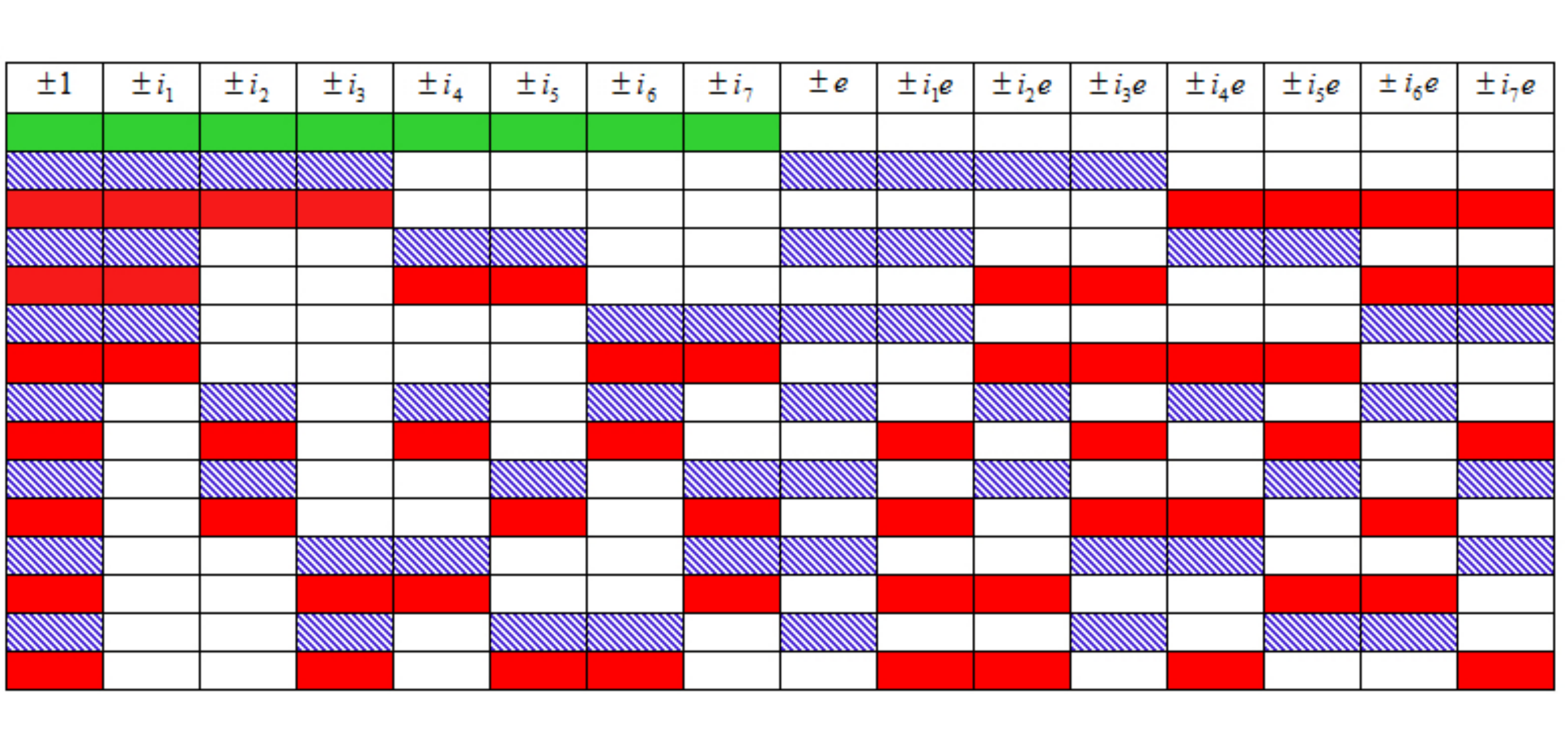}
		\caption{Subloops of $\mathbb{S}_{32}$ of index 2}
	\label{fig:SubInd2inSed}
\end{figure}
Next, we show that, starting at $\mathbb{S}_{32}$, any subloop of $Q_{n}$ of the third type is not a Cayley-Dickson loop. 
\begin{lemma}\label{lemma:aut3} Let $B \neq Q_{n-1}$ be a subloop of $Q_{n}$ of index 2 and $D$ be a subloop of $Q_{n-1}$ of index 2, $n\geq4$.
\begin{enumerate}
	\item \label{lemma:aut3-1} For any $x\in  Q_{n-1}$ there exist $y,z\in  Q_{n-1}$ such that $\left\langle x,y,z\right\rangle\cong\mathbb{O}_{16}$, $\{x,y,z\}\cap D\neq\emptyset$ and \\$\{x,y,z\}\cap\left(Q_{n-1}\backslash D \right)\neq\emptyset$.
	\item \label{lemma:aut3-2} If $e\notin B$ then for any $x\in B$ there exist $y,z\in B$ such that $\left\langle x,y,z\right\rangle\cong\tilde{\mathbb{O}}_{16}$.
	\item If $e\notin B$ then $B\ncong  Q_{n-1}$. In particular, any subloop of the third type is not a Cayley-Dickson loop.
\end{enumerate}
\end{lemma}
\begin{proof}
\begin{enumerate}
	\item The size of $D$ is $\frac{\left| Q_{n-1}\right|}{2}\geq 8$. Let $e\in  Q_{n-1}$. If $x\in D$, choose $y\notin D\cup \left\langle e,x\right\rangle$, then $\left\langle e,x,y\right\rangle\cong\mathbb{O}_{16}$ by Lemma~\ref{lemma:aut2}. Similarly, if $x\notin D$, choose $y\in D$, $y\notin \left\langle e,x\right\rangle$, then $\left\langle e,x,y\right\rangle\cong\mathbb{O}_{16}$ by Lemma~\ref{lemma:aut2}. If $x=e$, choose $y\notin D\cup\left\langle e\right\rangle$ and $z\in D\backslash \left\langle e,y\right\rangle$, then $\left\langle e,x,y\right\rangle\cong\mathbb{O}_{16}$ by Lemma~\ref{lemma:aut2}. 
	\item By Lemma~\ref{lemma:aut4}, $B=D\oplus\left(Q_{n-1}\backslash D\right)e$ for some subloop $D$ of $Q_{n-1}$ of index 2. Without loss of generality, suppose $x\in D$. By \ref{lemma:aut3-1} there exist $y,z\in  Q_{n-1}$ such that $\left\langle x,y,z\right\rangle\cong\mathbb{O}_{16}$, $\{x,y,z\}\cap D\neq\emptyset$ and $\{x,y,z\}\cap \left(Q_{n-1}\backslash D\right) \neq\emptyset$. Again, without loss of generality, suppose $y\in D$ and $z\in  Q_{n-1}\backslash D$, therefore $\left(x,0\right),\left(y,0\right),\left(z,1\right)\in B$. Using \eqref{eqn:oct1}, \eqref{eqn:oct2}, $\left\langle \left(x,0\right),\left(y,0\right),\left(z,1\right)\right\rangle\cong \tilde{\mathbb{O}}_{16}$.
	\item By Lemma~\ref{lemma:aut2}, there is an element $e\in Q_{n-1}$ such that for any $x,y\in  Q_{n-1}$, $\left|\left\langle e,x,y\right\rangle\right|=16$ implies that $\left\langle e,x,y\right\rangle\cong\mathbb{O}_{16}$. However, by \ref{lemma:aut3-2}, $B$ doesn't contain such an element. 
\qedhere	
\end{enumerate}
\end{proof}
\begin{lemma}\label{lemma:aut7}
Let $x\in  Q_{n}\backslash \left\{\pm1,\pm e\right\}$, $n\geq4$. There exist $y,z\in  Q_{n}$ such that $\left\langle x,y,z\right\rangle\cong\tilde{\mathbb{O}}_{16}$.
\end{lemma}
\begin{proof}
Without loss of generality, suppose $x\in  Q_{n-1}$. By Lemma~\ref{lemma:aut3} part~\ref{lemma:aut3-1}, there exist $y,z\in  Q_{n-1}$ such that $\left\langle x,y,z\right\rangle\cong\mathbb{O}_{16}$. Using \eqref{eqn:oct1}, \eqref{eqn:oct2}, $\left\langle \left(x,0\right),\left(y,0\right),\left(z,1\right)\right\rangle\cong \tilde{\mathbb{O}}_{16}$.
\end{proof}
On $Q_{n}$, define maps 
\begin{eqnarray}
(id,-id) & : & (x,x_{n+1})\mapsto((-1)^{x_{n+1}}x,x_{n+1}),\\
(id,id) & : & (x,x_{n+1})\mapsto (x,x_{n+1}),
\end{eqnarray} 
where $x\in  Q_{n-1}$ and $x_{n+1}\in\{0,1\}$. The map $(id,id)$ is an identity; the map $\phi=(id,-id)$ is an automorphism because 
\begin{eqnarray*}
\phi((x,0)(y,0)) & = & \phi((xy,0))=(xy,0)=(x,0)(y,0)=\phi((x,0))\phi((y,0)), \\
\phi((x,0)(y,1)) & = & \phi((yx,1))=(-yx,1)=(x,0)(-y,1)=\phi((x,0))\phi((y,1)), \\
\phi((x,1)(y,0)) & = & \phi((xy^{*},1))=(-xy^{*},1)=(-x,1)(y,0)=\phi((x,1))\phi((y,0)), \\ 
\phi((x,1)(y,1)) & = & \phi((-y^{*}x,0))=(-y^{*}x,0)=(-x,1)(-y,1)=\phi((x,1))\phi((y,1)).
\end{eqnarray*} 
\begin{proof}{(of Theorem \ref{thm:autom})}
Let $\phi: Q_{n}\mapsto  Q_{n}$, $n\geq 4$, be an automorphism. 
\begin{enumerate}
	\item By Proposition~\ref{prop:prop}, $\phi\left(1\right)=1$, $\phi\left(-1\right)=-1$.
	\item Let $x\in  Q_{n}\backslash\left\{\pm1,\pm e\right\}$. By Lemma~\ref{lemma:aut7}, there exist $y,z \in  Q_{n}$ such that $\left\langle x,y,z\right\rangle\cong \tilde{\mathbb{O}}_{16}$, however, by Lemma~\ref{lemma:aut2}, $\left\langle e,y,z\right\rangle\cong \mathbb{O}_{16}$ for any $y,z \in  Q_{n}$. Therefore it is only possible that $\phi\left(e\right)=e$, which holds when $\phi$ is an identity map, or $\phi\left(e\right)=-e$, which holds when $\phi=(id,-id)$.	   
	\item Consider the subloops of $Q_{n}$ of index 2. By Lemma~\ref{lemma:aut3}, any such subloop of the third type is not isomorphic to $Q_{n-1}$. A subloop of the first type (there is only one such subloop) is a copy of $Q_{n-1}$ in $Q_{n}$ of the form $\left\{(x,0)\left| \right.x\in  Q_{n-1}\right\}$. Therefore any element $(x,0)$ is contained in at least one more copy of $Q_{n-1}$ compared to an element $(y,1)$. This shows that for every $x \in  Q_{n-1}$,  $\phi\left((x,0)\right)=(y,0)$ for some $y \in  Q_{n-1}$ and hence $\psi\in Aut(Q_{n-1})$.
	\item Let $x\in Q_{n-1}$. Using multiplication formula \eqref{eqn:mult1}, $xe=(x,0)(1,1)=(x,1)$. If $\phi$ is an automorphism on $Q_{n}$, then $\phi((x,1))=\phi((x,0)(1,1))=\phi((x,0))\phi((1,1))=\psi(x)\phi(e)$.
	\qedhere
\end{enumerate} 
\end{proof}
Finally, we show that, starting at $\mathbb{S}_{32}$, $Aut(Q_{n})$ is a direct product of $Aut(Q_{n-1})$ and a cyclic group of order 2. 
\begin{theorem}\label{thm:autdirectproduct}
Let $Q_{n}$ be a Cayley-Dickson loop and let $n\geq 4$. Then $Aut\left(Q_{n}\right)\cong Aut\left(Q_{n-1}\right)\times \mathbb{Z}_{2}$.
\end{theorem}
\begin{proof}
Let $G=Aut\left(Q_{n}\right)$, $K=Aut\left(Q_{n-1}\right)$, $H=\left\{\left(id,id\right),\left(id,-id\right)\right\}\cong \mathbb{Z}_{2}$, $n\geq 4$.
\begin{enumerate}
	\item A group $K$ is normal in $G$ because $\left[G:K\right]=2$.  
	\item Next, show that $H$ is normal in $G$. Let $g\in G$, $h\in H$. Notice that $g^{-1}hg\in H$ iff $g^{-1}hg\upharpoonright_{Q_{n-1}}=id_{Q_{n-1}}$. Let $x\in  Q_{n-1},$ $g=kh_{0}$, where $k\in K$, $h_{0}\in H$.
	\begin{equation*} g^{-1}hg\left(x\right)=h_{0}^{-1}k^{-1}hk\underbrace{h_{0}\left(x\right)}_{x}=h_{0}^{-1}k^{-1}\underbrace{hk\left(x\right)}_{k\left(x\right)\in Q_{n-1}}=h_{0}^{-1}\underbrace{k^{-1}k\left(x\right)}_{x}=h_{0}^{-1}\left(x\right)=x,
	\end{equation*}
	therefore $g^{-1}hg\in H$.
	\item Both $K$ and $H$ are normal subgroups of $G$, therefore $KH\leq G$. Also, $\left|KH\right|\geq2\left|K\right|=\left|G\right|$, hence $KH=G$.
	\item Obviously, $(id,-id)\notin K$ and $H \cap K = id$.
	\qedhere
\end{enumerate}
\end{proof}
\paragraph{Acknowledgement} We thank Petr~Vojt\v{e}chovsk\'{y} for numerous discussions and suggestions. 
\newpage
\bibliography{cdlbib}
\bibliographystyle{plain} 
\end{document}